\documentclass[12pt,dvips]{article}

\usepackage{amsmath,amsfonts,amssymb,amsthm,graphicx,verbatim,subfig}
\usepackage[all]{xy}

\ifx\pdftexversion\undefined

\usepackage[a4paper,colorlinks,link
=black,filecolor=black,citecolor=black,urlcolor=black,pdfstartview=FitH]{hyperref}
\else

\usepackage[a4paper,colorlinks,linkcolor=black,filecolor=black,citecolor=black,urlcolor=black,pdfstartview=FitH]{hyperref}
\fi

\font\sixbb=msbm6
\font\eightbb=msbm8
\font\twelvebb=msbm10 scaled 1095
\newfam\bbfam
\textfont\bbfam=\twelvebb \scriptfont\bbfam=\eightbb
                           \scriptscriptfont\bbfam=\sixbb

\newtheorem{theorem}{\bf Theorem}[section]
\newtheorem{claim}[theorem]{\bf Claim}
\newtheorem{conjecture}[theorem]{\bf Conjecture}

\newtheorem{definition}[theorem]{\bf Definition}

\title{An upper bound on the number of high-dimensional permutations}
\begin{document}
\author{Nathan Linial\thanks{Department of Computer Science, Hebrew University, Jerusalem 91904,
    Israel. e-mail: nati@cs.huji.ac.il~. Supported by ISF and BSF grants.}
  \and {Zur Luria\thanks{Department of Computer Science, Hebrew University, Jerusalem 91904,
    Israel. e-mail: zluria@cs.huji.ac.il~.}}
}

\date{}

\maketitle
\pagestyle{plain}

\begin{abstract}

What is the higher-dimensional analog of a permutation? If we think of a permutation as
given by a permutation matrix, then the following definition suggests itself:
A $d$-dimensional
permutation of order $n$ is an $n \times n \times \ldots n = [n]^{d+1}$ array of zeros and
ones in which every {\em line} contains a unique $1$ entry. A line here is a set of entries
of the form $\{(x_1,\ldots,x_{i-1},y,x_{i+1},\ldots,x_{d+1}) | n \ge y \ge 1\}$ for some index
$d+1 \ge i \ge 1$ and some choice of $x_j \in [n]$ for all $j \neq i$. It is easy to observe
that a one-dimensional permutation is simply a permutation matrix and that a two-dimensional permutation
is synonymous with an
order-$n$ Latin square. We seek an estimate for the number of $d$-dimensional
permutations. Our main result is the following upper bound on their number
\[
 \left((1+o(1))\frac{n}{e^d}\right)^{n^d}.
\]
We tend to believe that this is actually the correct number, but the problem of proving the complementary
lower bound remains open. Our main tool is an adaptation of Br\`{e}gman's \cite{Br73} proof of the Minc conjecture on permanents.
More concretely, our approach is very close in spirit to Schrijver's~\cite{Sch78} and Radhakrishnan's \cite{Ra97} proofs of Br\`{e}gman's theorem.

\end{abstract}

\section{Introduction}
The permanent of an $n\times n$ matrix $A = (a_{ij})$ is defined by
\[ Per (A) = \sum_{\sigma \in \mathbb{S}_n}{\prod_{i=1}^n{a_{i,\sigma_i}}} \]
Permanents have attracted a lot of attention~\cite{Minc}.
They play an important role in combinatorics. Thus if $A$
is a 0-1 matrix, then $Per(A)$ counts perfect matchings in the bipartite graph whose adjacency matrix is $A$.
They are also of great interest from the computational perspective.
It is $\#P$-hard to calculate the permanent of a given 0-1 matrix \cite{Va79},
and following a long line of research, an approximation scheme was found~\cite{J+S+V}
for the permanents of nonnegative matrices. Bounds on permanents have also been studied at great depth.
Van der Waerden conjectured that $ Per(A) \geq \frac{n!}{n^n}$ for every $n \times n$ doubly stochastic matrix $A$, and 
this was established more than fifty years later by
Falikman and by Egorychev~\cite{Fal, Eg}. More recently, Gurvitz~\cite{Gur} discovered a
new conceptual proof for this conjecture (see~\cite{Lau+Sch} for a very
readable presentation). What is more relevant for us here are upper bounds on permanents.
These are the subject of Minc's
conjecture which was proved by Br\`{e}gman.

\begin{theorem}  If $A$ is an $n\times n$ 0-1 matrix with $r_i$ ones in the $i$-th row, then
$$ Per(A) \leq \prod_{i=1}^n{(r_i!)^{1/r_i}} . $$
\label{BregmanThm}
\end{theorem}

In the next section we review Radhakrishnan's proof, which uses the entropy method.
Our plan is to imitate this proof for a $d$-dimensional analogue of the permanent. 
To this end we need the notion of $d$-dimensional permutations.

\begin{definition} 
\begin{enumerate}
 \item Let $A$ be an $[n]^d$ array. A \textit{line} of $A$ is vector of the form 
  $$ ( A(i_1, ... ,i_{j-1},t,i_{j+1}, ... ,i_d) )_{t=1}^n ,$$
where $1 \le j \le d$ and $i_1, ... , i_{j-1}, i_{j+1} , ... ,i_d \in [n]$.
 \item A $d$-dimensional permutation of order $n$ is an $ [n]^{d+1}$ array $P$ of zeros and ones such that every line of 
$P$ contains a single one and $n-1$ zeros. Denote the set of all $d$-dimensional permutations of order $n$ by $S_{d,n}$.
\end{enumerate}
\end{definition}

For example, a two dimensional array is a matrix. It has two kinds of lines,
usually called rows and columns. Thus a $1$-permutation is an $n \times n$ 0-1 matrix with a single one in each row and a 
single one in each column, namely a permutation matrix. A $2$-permutation is identical to a Latin square
and $S_{2,n}$ is the same as the set ${\cal L}_n$, of order-$n$ Latin squares.
We now explain the correspondence between the two sets.
If $X$ is a $2$-permutation of order $n$, then we associate with it a Latin square $L$, where $L(i,j)$ as the (unique) 
index of a $1$ entry in the line $A(i,j, \ast)$. For more on the subject of Latin squares, see \cite{VL+W}.
The same definition yields a one-to-one correspondence
between $3$-dimensional permutations and Latin cubes. In general, $d$-dimensional 
permutations are synonymous with $d$-dimensional Latin hypercubes. For more on $d$-dimensional
Latin hypercubes, see \cite{Wa+Mc}.
To summarize, the following is an equivalent definition of a $d$-dimensional permutation.
It is an $[n]^d$ array with entries from $[n]$ in which every line contains each $i \in [n]$ exactly once.
We interchange freely between these two definitions according to context.

Our main concern here is to estimate $|S_{d,n}|$, the number of $d$-dimensional permutations of
order $n$.  By Stirling's formula

\[
 |S_{1,n}| = n! = \left((1+o(1)) \frac{n}{e} \right)^{n}.
\]
As we saw, $|S_{2,n}|$ is the number of order $n$ Latin squares. The best known estimate~\cite{VL+W} is
\[
  |S_{2,n}| = |\cal{L}\rm_n| = \left((1+o(1)) \frac{n}{e^2} \right)^{n^2}.
\]
This relation is proved using bounds on permanents. Br\'egman's theorem for the upper bound, and the Falikman-Egorychev
theorem for the lower bound.

This suggests
\begin{conjecture}
\label{main:conjecture}
 \[
  |S_{d,n}| = \left((1+o(1)) \frac{n}{e^d} \right)^{n^d}.
\]
\end{conjecture}

In this paper we prove the upper bound

\begin{theorem}
\label{mainResult}
 \[
  |S_{d,n}| \le \left((1+o(1)) \frac{n}{e^d} \right)^{n^d}.
\]
\end{theorem}

As mentioned, our method of proof is an adaptation of~\cite{Ra97}. We first need

\begin{definition} 
\begin{enumerate}
\item 
An $[n]^{d+1}$ 0-1 array $M_1$ is said to support an array $M_2$ if 
$$ M_2(i_1, ... ,i_{d+1}) = 1 \Rightarrow M_1(i_1, ... ,i_{d+1}) = 1 .$$ 
 \item The $d$-permanent of a $[n]^{d+1}$  0-1 array $A$ is
$$ Per_d(A) = \mbox{The number of $d$-dimensional permutations supported by $A$} .$$ 
\end{enumerate}
\end{definition}

Note that in the one-dimensional case, this is indeed the
usual definition of $Per(A)$. It is not hard to see
that for $d=1$ following theorem coincides with Br\`{e}gman's theorem.

\begin{theorem}
Define the function $f:\mathbb{N}_{\geq 0} \times \mathbb{N} \longrightarrow \mathbb{R}$ recursively by:
\begin{itemize}
 \item $f(0,r) = \log(r)$, where the logarithm is in base $e$.
 \item $f(d,r) = \frac{1}{r} \sum_{k=1}^r{f(d-1,k)}$.
\end{itemize}
Let $A$ be an $[n]^{d+1}$ 0-1 array with $r_{i_1 ,... i_d}$ ones in the line $A(i_1, ... ,i_d,\ast)$. Then
$$Per_d(A) \leq \prod_{i_1, ... ,i_d}{e^{f(d,r_{i_1 ,... i_d})}} .$$
\label{generalBregman}
\end{theorem}

We will derive below fairly tight bounds on the function $f$ that appears in theorem~\ref{generalBregman}.
It is then an easy matter to prove theorem \ref{mainResult} by applying theorem~\ref{generalBregman} to the all-ones array.

What about proving a matching lower bound on $S_{d,n}$ (and thus proving conjecture~\ref{main:conjecture})?
In order to follow the footsteps of~\cite{VL+W}, we would need a lower bound on $Per_d{A}$, namely,
a higher-dimensional analog of the van der Waerden conjecture. The entries of a {\em multi-stochastic} array
are nonnegative reals and the
sum of entries along every line is $1$. This is the higher-dimensional counterpart
of a doubly-stochastic matrix. It should be clear how to extend the notion of $Per_d(A)$ to real-valued arrays.
In this approach we would need a lower bound on $Per_d(A)$ that holds for every multi-stochastic array $A$.
However, this attempt (or at least its most simplistic version) is bound to fail.
An easy consequence of Hall's theorem says that a 0-1
matrix in which every line or column contains the same (positive)
number of 1-entries, has a {\em positive} permanent. (We still do not know exactly 
how small such a permanent can be, see~\cite{Lau+Sch} for more on this).
However, the higher dimensional analog of this is simply incorrect. There exist multi-stochastic arrays
whose $d$-permanent vanishes, as can easily be deduced e.g., from~\cite{Koch}.

We can, however, derive a lower bound of $|S_{d,n}| \ge \exp(\Omega(n^d))$ for even $n$. 
Consider the following construction: Let $n$ be an even integer, and let $P$ be a $d$-dimensional permutation 
of order $\left[ \frac{n}{2} \right] ^{d}$. It is easy to see that such a $P$ exists. Simply set 
$$ P(i_1, ... ,i_d) = (i_1 + ... + i_d) \mod \frac{n}{2} .$$
Now we construct a $d$-dimensional permutation $Q$ of order $[n]^d$ by replacing each element of $P$ with a $[2]^d$
block. If $P(i_1,...,i_d) = j$, then the corresponding block contains the values $j$ and 
$j+\frac{n}{2}$. It is easy to see that there are exactly two ways to arrange these values in each block, and that $Q$ is indeed a $d$-dimensional permutation
of order $[n]^d$. There are $\left( \frac{n}{2}\right)^d$ blocks, and so the number of possible $Q$'s is 
$2^{\left(\frac{n}{2}\right)^d}$. For a constant $d$ this is $\exp(\Omega(n^d))$. 

In section 2 we present Radhakrishnan's proof of the Br\`{e}gman bound. In section 3 we prove theorem \ref{generalBregman}.
In section 4 we use this bound to prove theorem \ref{mainResult}.  

\section{Radhakrishnan's proof of Br\`{e}gman's theorem}

\subsection{Entropy - Some basics}

We review the basic material concerning entropy that is used here
and refer the reader to~\cite{CT91} for further information on the topic.

\begin{definition} The \textit{entropy} of a discrete random variable $X$ is given by
$$ H(X) = \sum_x{\Pr(X=x) \log\left( \frac{1}{\Pr(X=x)} \right)} .$$
For random variables $X$ and $Y$, the \textit{conditional entropy of X given Y} is 
$$ H(X | Y) = \mathbb{E}[H(X|Y=y)] = \sum_y{\Pr(Y=y)H(X|Y=y)} .$$
\end{definition}
In this paper we will always consider the base $e$ entropy of $X$ which simply means that the logarithm is in base $e$.
\begin{theorem}
\begin{enumerate}
 \item If $X$ is a discrete random variable, then $$ H(X) \leq \log | range(X) | ,$$
       with equality iff $X$ has a uniform distribution.
 \item If $X_1, ... ,X_n$ is a sequence of random variables, then $$ H(X_1, ... ,X_n) = \sum_{i=1}^n{H(X_i|X_1, ... ,X_{i-1})} .$$
 \item The inequality $$ H(X|Y) \leq H(X|f(Y))$$ holds for every two discrete
random variables $X$ and $Y$ and every real function $f(\cdot)$. 
\end{enumerate}
\label{entropyThm}
\end{theorem}

The following is a general approach using entropy that is useful for a variety
of approximate counting problems. Suppose that we need to estimate the cardinality of
some set $S$. If $X$ is a random variable which takes values in $S$ under the uniform distribution on $S$,
then $H(X) = \log(|S|)$. So, a good estimate on $H(X)$ yields bounds on $|S|$.

This approach is the main idea of both Radhakrishnan's proof and our work.

\subsection{Radhakrishnan's proof}

Let $A$ be an $n \times n$ 0-1 matrix with $r_i$ ones in the $i$-th row.
Our aim is to prove the upper bound
$$Per(A) \leq \prod_{i=1}^n{(r_i !)^{\frac{1}{r_i}}} .$$

Let $\cal{M}$ be the set of permutation matrices supported by $A$, and let $X$ be a uniformly sampled random element of $\cal{M}$.
Our plan is to evaluate $H(X)$ using the chain rule and estimate $|\cal{M}|$
using the fact (theorem \ref{entropyThm}) that $H(X) = \log(|\cal{M}|)$.

Let $X_i$ be the unique index $j$ such that $X(i,j) = 1$. We consider a process where we scan the rows
of $X$ in sequence and estimate $H(X) = H(X_1, ... ,X_n)$ using the chain rule
in the corresponding order. To carry out this plan, we need to
bound the contribution of the term involving $X_i$ conditioned on the 
previously observed rows. That is, we write
$$ H(X) = \sum_{i=1}^n{H(X_i | X_1, ... ,X_{i-1})} .$$

Let $R_i$ be the set of indices of the 1-entries in $A$'s $i$-th row. That is, 
$$ R_i = \{ j:A(i,j)=1 \} .$$
Let 
$$ Z_i = \{ j \in R_i: X_{i'} = j \text{ for some } i'<i \} .$$
Note that $X_i \in R_i$, because $X$ is supported by $A$. In addition, given that we have already exposed the values $X_{i'}$ for
$i'<i$, it is impossible for $X_i$ to take any value $j \in Z_i$, or else
the column $X(\ast,j)$ contains more than a single $1$-entry. Therefore, given 
the variables that precede it, $X_i$ must take a value in $R_i \smallsetminus Z_i$. The cardinality $N_i = |R_i \smallsetminus Z_i|$
is a function of $X_1, ... ,X_{i-1}$ and so by theorem \ref{entropyThm}, 

$$ H(X) = \sum_{i=1}^n{H(X_i | X_1, ... ,X_{i-1} ) } $$ 
$$ = \sum_{i=1}^n{ \sum_{x_1, ... ,x_{i-1}}{\Pr(X_1 = x_1, ... ,X_{i-1}=x_{i-1}) H(X_i | X_1=x_1, ... ,X_{i-1}=x_{i-1} )} } $$
$$ \leq  \sum_{i=1}^n{ \sum_{x_1, ... ,x_{i-1}}{\Pr(X_1 = x_1, ... ,X_{i-1}=x_{i-1}) \log(N_i)} } $$
$$ = \sum_{i=1}^n{\mathbb{E}_{X_1, ... ,X_{i-1}}\left[\log(N_i) \right]}
 = \sum_{i=1}^n{\mathbb{E}_X\left[ \log(N_i) \right]} .$$

It is not clear how we should proceed from here, for how can we bound $\log(N_i)$ for a general matrix?
Moreover, different orderings of the rows will give different bounds. We use this fact to our advantage
and consider the expectation of this bound over all possible orderings. 
Associated with a permutation $\sigma \in \mathbb{S}_n$ is an ordering of the rows
where $X_j$ is revealed before $X_i$ if $\sigma(j) < \sigma(i)$.
We redefine $Z_i$ and $N_i$ to take the ordering $\sigma$ into account. Let 
$$ Z_i(\sigma) =  \{ j \in R_i: X_{i'} = j \text{ for some } \sigma(i') < \sigma(i) \} .$$ 
$$ N_i(\sigma) = |R_i \smallsetminus Z_i(\sigma)| .$$

Then $N_i(\sigma)$ is the number of available values for $X_i$, given all the variables 
$X_j$ for $j$ such that $\sigma(j) < \sigma(i)$.
As before, using the chain rule we obtain the inequality
$$ H(X) = \sum_{i=1}^n{H(X_i | X_j: \sigma(j) < \sigma(i))} \leq \sum_{i=1}^n{\mathbb{E}_X\left[\log (N_i(\sigma))\right]} .$$

The inequality remains true if we take the expected value of both sides when 
$\sigma$ is a random permutation sampled from the uniform distribution on $\mathbb{S}_n$. 
$$ H(X) \leq  \sum_{i=1}^n{\mathbb{E}_{\sigma}\left[\mathbb{E}_X\left[\log (N_i(\sigma))\right]\right]} = 
\sum_{i=1}^n{\mathbb{E}_X\left[\mathbb{E}_{\sigma}\left[\log (N_i(\sigma))\right]\right]}.$$
Thus, the bound we get on $H(X)$ depends on the distribution of the random variable $N_i(\sigma)$.
The final observation that we need is
that the distribution of $N_i(\sigma)$ is very simple and that it does not depend on $X$.
Consequently we can eliminate the step of taking 
expectation with respect to the choice of $X$. Let us fix a specific $X$.

Let $W_i$ denote the set of $r_i - 1$ row indices $j \ne i$ for which $X_j \in R_i$.
Note that $N_i$ is equal to $r_i$ minus the number of indices in $W_i$ that precede $i$ in the random ordering $\sigma$. 
Since $\sigma$ was chosen uniformly, this number is distributed uniformly in $\{ 0, ... ,r_i-1 \}$. 
Thus, $N_i$ is uniform on the set $\{ 1, ... , r_i \}$. Therefore
$$ \mathbb{E}_{\sigma}\left[\log (N_i(\sigma))\right] = \sum_{k=1}^{r_i}{\frac{1}{r_i} \log (k)} = \frac{1}{r_i} \log(r_i!) .$$
Hence
$$ H(X) \leq \sum_{i=1}^n{\mathbb{E}_{X}\left[\frac{1}{r_i} \log(r_i!) \right]} =
\sum_{i=1}^n{ \frac{1}{r_i} \log(r_i!) }$$
which implies the Br\`{e}gman bound.

\section{The d-dimensional case}

\subsection{An informal discussion}
 
The core of the above-described proof of the Br\`{e}gman bound can be viewed as follows.
Let us pick first a $1$-permutation $X$ that is contained in the matrix $A$
and consider the set $R_i$ of the $r_i$ $1$-entries in $A$'s $i$-th row. There are exactly $r_i$ indices $j$
for which $X_j \in R_i$. The random ordering of the rows determines
which of these will precede the $i$-th row (or will {\em cast its shadow} on the $i$-th row). 
The random number $u_i$ of rows that cast a shadow on the $i$-th row is uniformly distributed
in the range $\{0,\ldots,r_i-1\}$.
The contribution of this row to the upper bound on $H(X)$ is $\mathbb{E}_{\sigma}[\log N_i]$, where $N_i = r_i - u_i$ is the number of $1$-entries in the $i$-th row that are still unshaded. The expectation of $\log N_i$ is exactly $\frac{1}{r_i} \sum_{j=1}^{r_i} \log j = \frac{1}{r_i} \log(r_i!)$. 

How should we modify this argument to deal with $d$-dimensional permutations? We fix a $d$-dimensional
permutation $X$ that is contained in $A$ and consider a random ordering of all lines of the form $A(i_1, ... ,i_d,\ast)$.
Given such an ordering, we use the chain rule to derive an upper bound on $H(X)$. Each ordering yields a different bound. However,
as in the one dimensional case, the key insight is that averaging over all possible orderings (in a class that we
later define) gives us a simple bound on $H(X)$.

The overall structure of the argument remains the same. We consider a concrete line $A(i_1, ... ,i_d,\ast)$. 
Its contribution to the estimate of the entropy is $\log N$ where $N$ is the number of $1$-entries that remain unshaded
at the time (according to the chosen ordering) at which we compute the corresponding term in the chain rule for the entropy.
However, now shade can fall from $d$ different directions. The contribution of the line to the entropy will be the 
expected logarithm of the number of ones that remain unshaded after each of the $d$ dimensions has cast its shade on it.

The lines are ordered by a random lexicographic ordering.
At the coarsest level lines are ordered according to their first coordinate $i_1$. This ordering is chosen uniformly from $\mathbb{S}_n$. To understand how many $1$'s remain unshaded in a given line, we first consider the shade along the first coordinate. If it initially has $r$ 1-entries,
then the number of unshaded 1-entries after this stage is uniformly distributed on $[r]$. We then recurse with the remaining $1$-entries and proceed on the subcube
of codimension $1$ that is defined by the value of the first coordinate.
It is not hard to see how the recursive expression for $f(d,r)$ reflects this calculation. 

\subsection{In detail}

Let $A$ be a $[n]^{d+1}$-dimensional array of zeros
and ones, and $X$ is a random $d$-permutation sampled uniformly from the set of $d$-permutations contained in $A$. 
Then $H(X) = \log(Per_d(A))$ by theorem \ref{entropyThm} and again we seek an upper bound on $H(X)$.

We think of $X$ as an $[n]^d$ array each line of which contains each member of $[n]$ exactly once.
The proof does its accounting using lines of the form $A(i_1, ... ,i_d,\ast)$, i.e.,
lines in which the $(d+1)$-st coordinate varies. Such a line is
specified by $\textbf{i} = (i_1, ... ,i_d)$. The random variable $X_{\textbf{i}}$ is defined to be
the value of $ X(i_1, ... ,i_d) $. We think of the variables $X_{\textbf{i}}$ as being revealed to us one by one.
Thus, $X_{i_1, ... ,i_d}$ must belong to  
$$ R_{\textbf{i}} = R_{i_1 , ... ,i_d} = \{ j: A(i_1, ... ,i_d,j) = 1 \}$$
the set of $1$-entries in this line.

In the proof we scan these lines in a particular randomly chosen order. Let us ignore
this issue for a moment and consider some fixed ordering of these lines. Initially, the
number of $1$-entries in this line is $r_{\textbf{i}}$. As we proceed, some of these $1$'s
become unavailable to $X_{\textbf{i}}$, since choosing them would result in a conflict with
the choice made in some previously revealed line.
We say  that these $1$'s are {\em in the shade} of previously considered lines.
This shade can come from any of the $d$ possible {\em directions}. Thus we denote by
$Z_{\textbf{i}} \subseteq R_{\textbf{i}}$ the set of the indices of the $1$-entries in $R_{\textbf{i}}$ that
are unavailable to $X_{\textbf{i}}$ given the values of the preceding variables.
We can express $Z_{\textbf{i}}=\cup_{k=1}^{d} Z^k_{\textbf{i}}$ where entries in $Z^k_{\textbf{i}}$ are shaded from
direction $k$. Namely, a member $j$ of $R_{\textbf{i}}$ belongs to $Z^k_{\textbf{i}}$ if there is an already scanned line indexed by ${\textbf{i'}}$
with $X_{\textbf{i'}}=j$ and where ${\textbf{i}}$ and  ${\textbf{i'}}$
coincide on all coordinates except the $k$-th.
Thus, given the values of the previously considered variables, there are at most
$$ N_{\textbf{i}} = |R_{\textbf{i}} \smallsetminus Z_{\textbf{i}}|$$
values that are available to $X_{\textbf{i}}$.

We next turn to the random ordering of the lines. 
Now, however, we do not select a completely random ordering,
but opt for a random lexicographic ordering. Namely, we select $d$ random permutations $\sigma_1, ... ,\sigma_d \in \mathbb{S}_d$. The line $A(i_1, ... ,i_d,\ast)$ precedes $A(i_1', ... ,i_d',\ast)$ if there is a $ k \in [n]$ such that $\sigma_k(i_k) < \sigma_k(i_k') $ and $i_j = i_j'$ for all $j < k$. Thus a choice of the orderings $\sigma_k$ induces a total order on the lines $A(i_1, ... ,i_d,\ast)$.
Denote this order by $\prec$. That is, we write $\textbf{i} \prec \textbf{j}$ if $\textbf{i}$ comes before 
$\textbf{j}$. We write $\textbf{i} \prec_k \textbf{j}$ if $\textbf{i} \prec \textbf{j}$ and $\textbf{i}$ and $\textbf{j}$
differ {\em only} in the $k$-th coordinate.

We think of $X_{\textbf{i}}$ as being revealed to us according to this order.

We turn to the definition of $R_{\textbf{i}}$, $Z^{k}_{\textbf{i}}$ and $N_{\textbf{i}}$.
Their definitions are affected by the chosen ordering of the lines. 
In addition, for reasons to be made clear later, we generalize the definition 
of $N_{\textbf{i}}$. It is defined as the number of values available to $X_{\textbf{i}}$ 
(given the preceding lines) from a given 
index set $W \subseteq R_{\textbf{i}}$. 
In the discussion below, we fix $X$, a $d$-dimensional permutation that is contained in $A$.

\begin{definition}
The index set of the $1$-entries in the line $A(i_1, ... ,i_d,\ast)$ is denoted by
$$ R_{\textbf{i}} = R_{i_1 , ... ,i_d} = \{ j: A(i_1, ... ,i_d,j) = 1 \} ,$$
and its cardinality is $r_{\textbf{i}} = |R_{\textbf{i}}| $.\\

Let $W \subseteq R_{\textbf{i}}$ with $\textbf{i} = (i_1, ... ,i_d)$, and suppose that 
$X_{\textbf{i}} \in W$. For a given ordering $\prec$, let 
$$ Z^{k}_{\textbf{i}}(X,\prec) = \{ j \in R_{\textbf{i}}: X_{\textbf{i'}} = j 
\text{ for some } \textbf{i'} \prec_k \textbf{i} \} .$$
$$ N_{\textbf{i}}(W,X,\prec) = | W \smallsetminus \cup_{k=1}^d{Z_{\textbf{i}}^k(X,\prec)} | .$$
\end{definition}

Thus, $N_{\textbf{i}}$ is a function of $W \subseteq R_{\textbf{i}}$, $X$ and the ordering $\prec$.
Each variable $X_{\textbf{i}}$ specifies a $1$ entry of the line $A(i_1, ... ,i_d, \ast)$.
The entry thus specified must conform to the values taken
by the preceding variables. Namely, no line of $X$ can contain more than a single $1$ entry.
We consider the number of values that the variable $X_{\textbf{i}}$ can take, given the values that precede it.
Fix an index tuple $\textbf{i} = (i_1, ... ,i_d)$. The variable $X_{\textbf{i}}$
must specify an index $i_{d+1}$ with 
$A(i_1, ... ,i_{d+1})=1$, i.e., an element of $R_{\textbf{i}}$. Consider some element $j \in R_{\textbf{i}}$.
If $X_{\textbf{i'}} = j$, for some $\textbf{i'} \prec_k \textbf{i}$ and $k \leq d$ then clearly $X_{\textbf{i}} \ne j$, or
else the line $X(i_1, ... ,i_{k-1},\ast,i_{k+1}, ... ,i_d)$ contains more than a single $j$-entry. 
In other words, $X_{\textbf{i}}$ cannot specify an element of $Z^{k}_{\textbf{i}}(X,\prec)$ and is restricted to the set
$R_{\textbf{i}} \smallsetminus \cup_{k=1}^d{Z^{k}_{\textbf{i}}(X,\prec)}$. Therefore, there are at most 
$N_{\textbf{i}}(R_i,X,\prec)$ 
possible values that $X_{\textbf{i}}$ can take given the variables that precede it in the order $\prec$. 

For a given order $\prec$, we can use the chain rule to derive
$$ H(X) = \sum_{\textbf{i}}{H(X_{\textbf{i}}|X_{\textbf{j}} : \textbf{j} \prec \textbf{i})} .$$ 

By theorem \ref{entropyThm},

$$ H(X_{\textbf{i}}|X_{\textbf{j}} : \textbf{j} \prec \textbf{i}) = 
\mathbb{E}_{X_{\textbf{j}} : \textbf{j} \prec \textbf{i}}
\left[H(X_{\textbf{i}} | X_{\textbf{j}} = x_{\textbf{j}}: \textbf{j} \prec \textbf{i})\right] $$
$$ \leq \mathbb{E}_{X_{\textbf{j}} : \textbf{j} \prec \textbf{i}}
\left[\log(N_{\textbf{i}}(R_{\textbf{i}},X,\prec))\right]  
= \mathbb{E}_X\left[\log(N_{\textbf{i}}(R_{\textbf{i}},X,\prec))\right] .$$

The last equality holds because $N_{\textbf{i}}$ depends only on the lines of 
$X$ that precede $X_{\textbf{i}}$, and so taking the expectation over the rest of $X$ doesn't change anything.

As in the one dimensional case, the next step is to take the expectation of both sides of the above inequality over $\prec$.
$$H(X) \leq \sum_{\textbf{i}}{\mathbb{E}_{\prec} \left[\mathbb{E}_X 
\left[\log(N_{\textbf{i}}(R_{\textbf{i}},X,\prec))\right]\right] } $$
$$ = \sum_{\textbf{i}}{\mathbb{E}_X \left[\mathbb{E}_{\prec} 
\left[\log(N_{\textbf{i}}(R_{\textbf{i}},X,\prec))\right]\right] }  .$$ 
The key to unraveling this expression is the insight that the random variable 
$N_{\textbf{i}}$ has a simple distribution (as a function of $\prec$), and moreover, that this distribution does not depend on $X$.

Recall that in the one dimensional case, we obtained the distribution of $N_i$ as follows. Initially, the number of ones in the 
$i$-th row was $r_i$. Then the rows preceding the $i$-th row were revealed, and some of the ones in the $i$-th row became 
unavailable to $X$, because some other row had placed a one in their column. We defined 
$N_i = |R_i \smallsetminus Z_i(\sigma)|$. The size of $Z_i(\sigma)$ was shown to be uniformly distributed over 
$\{0, ... ,r_i-1\}$, and thus the distribution of $N_i$ was shown to be uniform over $\{1, ... ,r_i\}$.

A similar argument works in the $d$ dimensional case, but the distribution of $N_{\textbf{i}}$ 
is no longer uniform. Recall that the function $f$ is defined recursively by
$$ f(0,r) = \log(r)$$ $$f(d,r) = \frac{1}{r}\sum_{k=1}^r{f(d-1,k)}.$$

\begin{claim}
Let $X$ be a $d$-permutation, $\textbf{i} = (i_1, ... ,i_d)$ and let $W \subseteq R_{\textbf{i}}$ be an index set such that $X_{\textbf{i}} \in W$. Then
$\mathbb{E}_{\prec}\left[ \log(N_{\textbf{i}}(W,X,\prec)) \right]$ depends only on $d$ and $r = |W|$, and
$$ \mathbb{E}_{\prec}\left[ \log(N_{\textbf{i}}(W,X,\prec)) \right] = f(d,r) .$$
\label{fStructure}
\end{claim}

\begin{proof}

The proof proceeds by induction on $d$. 

First, note that if $|W| = r$ and $d=0$, then $N_{\textbf{i}}(W,X,\prec) = |W| = r$ by definition, and therefore 
$$\mathbb{E}_{\prec}\left[ \log(N_{\textbf{i}}(W,X,\prec)) \right] = \log(r) = f(0,r) .$$ 

In order to proceed with the induction step, we must describe $N_{\textbf{i}}(W,X,\prec)$ in terms of parameters of dimension $d-1$ 
instead of $d$. To this end we need the following definitions:

\begin{itemize}
 \item $X' = X(i_1, \ast, ... ,\ast)$. Note that $X'$ is a $(d-1)$-dimensional permutation.	
 \item $W' = W \smallsetminus Z_{\textbf{i}}^1(X,\prec)$. Note that $|W'|$ actually depends only on $\sigma_1$, 
the ordering of the first coordinate.
 \item Let $\textbf{i'} = ( i'_1, ... ,i'_{d-1} ) = ( i_2, ... ,i_d )$.
 \item Given an ordering $\prec$, let $\prec'$ be the ordering on the index tuples $( i'_1, ... ,i'_{d-1} )$ defined by the orderings $ \sigma_2, \sigma_3, ... ,\sigma_d $.
\end{itemize}

Note that for every $X$,$W$, $\textbf{i}$ and $\prec$ we have $N_{\textbf{i}}(W,X,\prec) = N_{\textbf{i'}}(W',X',\prec')$. This 
equality follows directly from the definition of $N$. Now,
$$ \mathbb{E}_{\prec}\left[ \log(N_{\textbf{i}}(W,X,\prec)) \right] 
 = \mathbb{E}_{\sigma_1} \left[ \mathbb{E}_{\prec'}\left[ \log(N_{\textbf{i}}(W,X,\prec)) \right] \right]$$
$$ =  \mathbb{E}_{\sigma_1} \left[ \mathbb{E}_{\prec'}\left[ \log(N_{\textbf{i'}}(W',X',\prec')) \right] \right]
= \mathbb{E}_{\sigma_1} \left[ f(d-1,|W'|) \right] $$
The last step follows from the induction hypothesis. Consequently,
$$ \mathbb{E}_{\prec}\left[ \log(N_{\textbf{i}}(W,X,\prec)) \right] = \sum_k{\Pr (|W'|=k) f(d-1,k)} .$$

The only remaining question is to determine the distribution of $|W'|$ as a function of $\sigma_1$.
Note, however, that we have already answered this question 
in the one dimensional proof, namely, $|W'|$ is uniformly distributed on $\{ 1, ... ,r \}$. Indeed, 
$W' = |W \smallsetminus Z_{\textbf{i}}^1(X,\prec)|$, and $Z_{\textbf{i}}^1(X,\prec)$ 
is the set of indices $s$ such that: 
\begin{itemize}
 \item For some $j \in W$ , $X(s,i_2, ... ,i_d) = j$ (there are $r-1$ such indices, one for each $j \in W$).
 \item The random ordering $\sigma_1$ places $s$ before $i_1$.
\end{itemize}
In a random ordering, the position of $i_1$ is uniformly distributed.
Therefore $|Z_{\textbf{i}}^1(X,\prec)|$ is uniformly distributed on $\{0, ... ,r-1\}$, and 
$\Pr(|W'|=k) = \frac{1}{r}$ for every $1 \leq k \leq r$.

Putting this together, we have shown that 
$$   \mathbb{E}_{\prec}\left[ \log(N_{\textbf{i}}(W,X,\prec)) \right] = \frac{1}{r} \sum_{k=1}^r{f(d-1,k)} = f(d,r) .$$

\end{proof}

In conclusion, we have shown that 
$$H(X) \leq \sum_{\textbf{i}}{\mathbb{E}_X \left[\mathbb{E}_{\prec} 
\left[\log(N_{\textbf{i}}(R_{\textbf{i}},X,\prec))\right]\right] } $$
$$ = \sum_{\textbf{i}}{\mathbb{E}_X \left[f(d,r_{\textbf{i}})\right]} 
= \sum_{\textbf{i}}{f(d,r_{\textbf{i}})},$$ where $r_{\textbf{i}} = r_{i_1, ... ,i_d}$
is the number of ones in the vector $A(i_1, ... ,i_d,\ast)$. Therefore,
$$ Per_d(A) \leq \prod_{\textbf{i}}{e^{f(d,r_{\textbf{i}})}} .$$

\section{The number of d-permutations -- An upper bound}

As mentioned, the upper bound on the number of
$d$-dimensional permutations is derived by applying theorem~\ref{generalBregman} to the all-ones array $J$.
The main technical step is a derivation of an upper bound on the function $f(d,r)$.

\begin{theorem}
For every $d$ there exist constants $c_d$ and $r_d$ such that for all $r\geq r_d$,
$$ f(d,r) \leq \log(r) - d + c_d \frac{\log^d(r)}{r} .$$
One possible choice that we adopt here is $r_d = e^d$ for every $d$, $c_1 = 5$, $c_2 = 8$, and 
$c_d = \frac{d^3 (1.1)^d}{d!}$ for $d \geq 3$.
\end{theorem}
\begin{proof}
A straightforward induction on $d$
yields the weaker bound $f(d,r) \leq \log(r)$ for all $d,r$. For $d=0$ there is
equality and the general case follows since 
$ f(d,r) = \frac{1}{r}\sum_{k=1}^r{f(d-1,k)} \leq \frac{1}{r}\sum_{k=1}^r{\log(k)} \leq \log(r).$
This simple bound serves us to deal with the range of small $r$'s (below $r_{d-1}$).
We turn to the main part of the proof

$$ f(d,r) = \frac{1}{r}\sum_{k=1}^r{f(d-1,k)} = \frac{1}{r}
\left[ \sum_{k=1}^{r_{d-1}}{f(d-1,k)} + \sum_{k=r_{d-1}+1}^{r}{f(d-1,k)} \right] $$
$$\leq \frac{1}{r} \left[ r_{d-1} \log(r_{d-1}) + \sum_{k=1}^{r}{\log(k) - (d-1) + c_{d-1} \frac{\log^{d-1}(k)}{k}} \right] $$
$$ \leq \frac{\xi}{r} + \frac{1}{r}\log(r!) - (d-1) + \frac{c_{d-1}}{r}\sum_{k=1}^r{ \frac{\log^{d-1}(k)}{k} } $$
where $\xi = r_{d-1} \log(r_{d-1}) = (d-1) e^{d-1}$. It is easily verified that for $r \ge r_d \ge 3$
there holds $\log(r!) \leq r \log(r) - r + 2 \log(r)$. We can proceed with
$$ \leq \frac{\xi}{r} + \log(r) + \frac{2 \log(r)}{r} - d + \frac{c_{d-1}}{r}\sum_{k=1}^r{ \frac{\log^{d-1}(k)}{k} }. $$

We now bound the sum $\sum_{k=1}^r{ \frac{\log^{d-1}(k)}{k} }$ by means
of the integral $ \int_1^r{\frac{\log^{d-1}(x)dx}{x}} = \frac{\log^d(r)}{d} $. 
Note that the integrand is unimodal and its maximal value is
$\gamma = \left( \frac{d-1}{e} \right)^{d-1}$.
Thus,
$$ \frac{c_{d-1}}{r}\sum_{k=1}^r{ \frac{\log^{d-1}(k)}{k} } \leq \frac{c_{d-1}}{r} \left( \frac{\log^d(r)}{d} + \gamma \right) .$$
Putting this together, we have the inequality
$$f(d,r) \leq \log(r) - d + \frac{2 \log(r) + \xi + c_{d-1} \left(\gamma + \frac{\log^d(r)}{d} \right) }{r} .$$
Therefore it is sufficient to choose $c_d$ such that for every $r \geq e^d$
$$  2 \log(r) + \xi + c_{d-1} \left(\gamma + \frac{\log^d(r)}{d} \right) \leq c_d \log^d(r) $$
i.e.,
$$ \frac{2}{\log^{d-1}(r)} + \frac{\xi}{\log^d(r)} + 
c_{d-1} \left(\frac{\gamma}{\log^d(r)} + \frac{1}{d} \right) \leq c_d.$$
The left hand side of the above inequality is clearly a decreasing function of $r$. Therefore it is sufficient to verify 
the inequality for $r = e^d$. Plugging this and the values of the constants $\xi$ and $\gamma$ into the left hand 
side of the above inequality, we get
$$ \frac{2}{d^{d-1}} + \frac{(d-1) e^{d-1}}{d^d} + c_{d-1} \left( \frac{(d-1)^{d-1}}{e^{d-1} d^d} + \frac{1}{d}\right)  $$
$$ \leq \left( 1 + \frac{1}{e^{d-1}} \right) \frac{ c_{d-1} }{d} +
d \left( \frac{2}{d^d} + \left(\frac{e}{d}\right)^d \right) .$$
Thus, we may take 
$$ c_d = \left( 1 + \frac{1}{e^{d-1}} \right) \frac{ c_{d-1} }{d} +
d \left( \frac{2}{d^d} + \left(\frac{e}{d}\right)^d \right) . $$
Calculating $c_d$ using this recursion and the fact that $c_0 = 0$, we get that $c_1 = 2 + e \leq 5$, $c_2 \leq 8 $, and $c_d \leq \frac{d^3 (1.1)^d}{d!}$ for $3\leq d \leq 10$. 
Proceeding by induction,
$$ c_d = \left( 1 + \frac{1}{e^{d-1}} \right) \frac{ (d-1)^3 (1.1)^{d-1} }{d!} + 
d \left( \frac{2}{d^d} + \left(\frac{e}{d}\right)^d \right) $$
$$ \leq \frac{(1.1)^d (d-1)^3 }{d!} + 2 d \left(\frac{e}{d}\right)^d \leq \frac{(1.1)^d (d-1)^3 + 2 d^2}{d!} \leq 
\frac{(1.1)^d d^3}{d!} . $$
In the inequality before the last one, we used the fact that for $d \geq 10$ , $  \left( \frac{e}{d} \right)^d \leq \frac{d}{d!} $.

\end{proof}

For the $[n]^{d+1}$ all ones array $J$, $r_{i_1, ... ,i_d} = n$ for every tuple $(i_1, ... ,i_d)$, and so for large enough $n$ 
we have the bound
$$ Per_d(J) \leq \prod_{i_1, ... ,i_d}{e^{f(d,n)}} = \left(e^{f(d,n)}\right)^{n^d} \leq 
\left(  \exp\left[\log(n) - d + c_d \frac{\log^d(n)}{n}\right] \right)^{n^d}. $$
For a constant $d$, letting $n$ go to infinity, $c_d \frac{\log^d(n)}{n} = o(1)$ and therefore the number of $d$ permutations is at most 
$$\left((1 + o(1)) \frac{n}{e^d} \right)^{n^d} .$$

\end{document}